\documentclass[12pt]{amsart}
\usepackage{amsfonts}
\usepackage{amsfonts,latexsym,rawfonts,amsmath,amssymb,amsthm}
\usepackage[plainpages=false]{hyperref}
\usepackage{graphicx}

\numberwithin{equation}{section}

\RequirePackage{color}
\theoremstyle{plain}

\newtheorem{theorem}{Theorem}[section]

\newtheorem{conjecture}{Conjecture}[section]
\newtheorem{corollary}{Corollary}[section]
\newtheorem{definition}{Definition}[section]

\newtheorem{lemma}[theorem]{Lemma}
\newtheorem{proposition}[theorem]{Proposition}
\newtheorem{remark}[theorem]{Remark}

\newcommand{\beq}{\begin{equation}}
\newcommand{\eeq}{\end{equation}}
\newcommand{\beqs}{\begin{eqnarray*}}
\newcommand{\eeqs}{\end{eqnarray*}}
\newcommand{\beqn}{\begin{eqnarray}}
\newcommand{\eeqn}{\end{eqnarray}}
\newcommand{\beqa}{\begin{array}}
\newcommand{\eeqa}{\end{array}}

\def\phi{\varphi}

\begin{document}
\title[Prescribed Weingarten curvatures]{$k$-convex hypersurfaces with prescribed Weingarten curvature in warped product manifolds}

\author{Xiaojuan Chen}
\address{Faculty of Mathematics and Statistics, Hubei Key Laboratory of Applied Mathematics, Hubei University,  Wuhan 430062, P.R. China}
\email{201911110410741@stu.hubu.edu.cn}

\author{Qiang Tu$^{\ast}$}
\address{Faculty of Mathematics and Statistics, Hubei Key Laboratory of Applied Mathematics, Hubei University,  Wuhan 430062, P.R. China}
\email{qiangtu@hubu.edu.cn}

\author{Ni Xiang}
\address{Faculty of Mathematics and Statistics, Hubei Key Laboratory of Applied Mathematics, Hubei University,  Wuhan 430062, P.R. China}
\email{nixiang@hubu.edu.cn}



\thanks{This research was supported by funds from the National Natural Science Foundation of China No. 11971157, 12101206.}
\thanks{$\ast$ Corresponding author}

\begin{abstract}
In this paper, we consider  Weingarten curvature equations
for $k$-convex hypersurfaces with $n<2k$ in a warped product manifold  $\overline{M}=I\times_{\lambda}M$.
Based on the conjecture proposed by Ren-Wang in \cite{Ren2}, which is valid for $k\geq n-2$, we derive curvature estimates for equation $\sigma_k(\kappa)= \psi (V, \nu (V))$ through a straightforward proof.
Furthermore, we also obtain
an existence result for the star-shaped compact hypersurface $\Sigma$ satisfying the above equation by the degree theory under some sufficient conditions.
\end{abstract}

\maketitle

\baselineskip18pt

\parskip3pt

 \section{Introduction}

 Let $(M,g')$ be a compact Riemannian manifold and $I$ be an open interval in $\mathbb{R}$. The warped product manifold $\overline{M}=I\times_{\lambda}M$ is endowed with the metric
\begin{eqnarray}\label{metric}
\overline{g}^2=dr^2+\lambda^2(r) g',
\end{eqnarray}
where $\lambda:I\rightarrow\mathbb{R}^{+}$ is a positive $C^2$ differential function.
Let $\Sigma$ be a  compact star-shaped hypersurface in $\overline{M}$, thus $\Sigma$ can be parametrized as a radial graph over $\overline{M}$. Specifically speaking,
there exists a differentiable function
$r : M \rightarrow I$ such that the  graph of $\Sigma$ can be represented by
\begin{equation*}
  \Sigma=\{X(u)=(r(u),u)\mid u\in{M}\}.
\end{equation*}


In this paper, we consider the following prescribed Weingarten curvature equation in warped product
manifold $\overline{M}$
\begin{equation}\label{Eq}
\sigma_k(\kappa(V))=\psi(V, \nu(V)), \quad \forall~ V\in \Sigma,
\end{equation}
where $V=\lambda\frac{\partial}{\partial r}$ is the position vector field of hypersurface $\Sigma$ in $\overline{M}$, $\sigma_{k}$ is the $k$-th elementary symmetric function,
$\nu(V)$ is the outward unit normal vector field along the hypersurface
$\Sigma$ and $\kappa(V)=(\kappa_1,\cdots,\kappa_n)$ are the principle
curvatures of hypersurface $\Sigma$ at $V$.

 Curvature estimates for
 equation \eqref{Eq} in $\mathbb{R}^{n+1}$ has been studied extensively.  When $k=1$ and  $k=n$, the equation is quasi-linear equation and Gauss curvature equation respectively, then the corresponding curvature estimates follow from the classical theory of quasi-linear PDEs and Monge-Amp\`ere type equations in  \cite{CNS}.
When $\psi$ is independent of $\nu$, curvature estimates were proved by Caffarelli-Nirenberg-Spruck \cite{Ca}
for a general class of fully nonlinear operators $F$, including
$F=\sigma_k$ and $F=\frac{\sigma_k}{\sigma_l}$. When $\psi$ depends only on $\nu$, curvature estimates were proved by Guan-Guan \cite{Guan02}.
Curvature estimates were also proved for equation of prescribing curvature
measures problem in \cite{Guan12, Guan09}, where $\psi(X, \nu) = \langle X, \nu \rangle \Tilde{\psi}(X)$.
Ivochkina \cite{Iv1,Iv2} considered the Dirichlet problem of equation \eqref{Eq} and obtained curvatute estimates under some extra conditions on the dependence of $\psi$ on $\nu$.

In recent years, there are many progresses on establishing curvature estimates for equation \eqref{Eq} in case $2\leq k \leq n-1$.
 When $k=2$, curvature estimates for admissible solutions of equation \eqref{Eq} were obtained by Guan-Ren-Wang \cite{Guan-Ren15}. They also established curvature estimates of convex solutions for general $k$, see a simpler proof in Chu \cite{Chu21}. Subsequently, Spruck-Xiao \cite{Sp} extended $2$-convex case to space forms and gave a simple proof for the Euclidean case. In \cite{Ren, Ren1}, Ren-Wang proved curvature estimates for $k=n-1$ and $n-2$, respectively. They also proved curvature estimates for equation \eqref{Eq} with $n<2k$ in \cite{Ren2} based on a concavity conjecture.

Moreover, some results have
been obtained by Li-Oliker \cite{Li-Ol} on unit sphere, Barbosa-de
Lira-Oliker \cite{Ba-Li} on space forms, Jin-Li \cite{Jin} on
hyperbolic space, Andrade-Barbosa-de Lira \cite{An} on warped
product manifolds, Li-Sheng \cite{Li-Sh} for Riemannain manifold
equipped with a global normal Gaussian coordinate system. In particular, Chen-Li-Wang \cite{Chen} generalized the results in \cite{Guan-Ren15, Ren} to $(n-1)$-convex hypersurfaces in warped product manifolds.

Inspired by the above works, it is natural to consider extending Ren-Wang's results in \cite{Ren, Ren1, Ren2} from Euclidean space to  warped product manifolds.
Here we introduce the following conjecture:

\begin{conjecture}\label{key}
Let $\kappa=(\kappa_1,\cdots,\kappa_n)\in\Gamma_k$ with $\kappa_1\geq\kappa_2\geq\cdots\geq\kappa_n$ and $n<2k$. Assume that there exist constants $N_0, N_1$ such that $N_0\leq\sigma_k(\kappa)\leq N_1$. If there exist constants $K$ and $B$ such that $\kappa_1\geq B$, then
\begin{equation*}
  \kappa_1\left(K(\sum_j\sigma_k^{jj}(\kappa)\xi_j)^2-\sigma_k^{pp,qq}(\kappa)\xi_p\xi_q\right)-\sigma_k^{11}(\kappa)\xi_1^2+\sum_{j\neq1}a_j
\xi_j^2\geq0,
\end{equation*}
for any $\xi=(\xi_1,\cdots,\xi_n)\in \mathbb{R}^n$. Here $a_j=\sigma_k^{jj}(\kappa)+(\kappa_1+\kappa_j)\sigma_k^{11,jj}(\kappa)$.
\end{conjecture}

The main theorem is as follows.

\begin{theorem}\label{Main}
Let $r_1$, $r_2$ be constants with $r_1<r_2$,  $M$ be a compact Riemannian manifold, $\overline{M}$ be the warped product manifold
with the metric (\ref{metric}) and $\Gamma$ be an open neighborhood of unit normal bundle of $M$ in $\overline{M}\times \mathbb{S}^n$.  Assume that $\lambda$ is a positive $C^2$ differential function with $\lambda'>0$ and Conjecture \ref{key} holds. Suppose $\psi$ satisfies\par
\begin{eqnarray}\label{ASS1}
\psi(V,\nu)>C_n^k\zeta^k(r)\quad \quad \forall~ r\leq r_{1},
\end{eqnarray}
\begin{eqnarray}\label{ASS2}
\psi(V,\nu)<C_n^k\zeta^k(r)\quad \quad \forall ~r \geq r_{2}
\end{eqnarray}
and
\begin{eqnarray}\label{ASS3}
\frac{\partial}{\partial r}(\lambda^k\psi(V,\nu))\leq 0 \quad \quad \forall~ r_{1}<r<r_{2},
\end{eqnarray}
where $V=\lambda\frac{\partial}{\partial r}$ and $\zeta(r)=\lambda'(r)/\lambda(r)$.
Then there exists a $C^{4, \alpha}$, $k$-convex, star-shaped and closed hypersurface $\Sigma$ in the annulus domain $\{(r,u)\in \overline{M}\mid r_{1}\leq r \leq r_{2}\}$ that satisfies equation \eqref{Eq} for any $\alpha\in (0,1)$.
\end{theorem}
\begin{remark}
The key to prove Theorem \ref{Main} is to obtain curvature estimates (Theorem \ref{n-2-C2e}) for
this Hessian type equation in warped product manifold. Compared to the proof in Euclidean space by Ren-Wang \cite{Ren2}, we give a  straightforward proof. Note that Conjecture \ref{key} is weaker than the one proposed by Ren-Wang in \cite{Ren, Ren1, Ren2}.
\end{remark}
It is worth noting that Conjecture \ref{key} holds for $k\geq n-2$, which was proved in Ren-Wang \cite{Ren, Ren1, Ren2}. Thus we can directly get the following results.
\begin{corollary}
Let $k\geq n-2$.  $M$,  $\overline{M}$, $\Gamma$, $\lambda$ and $\psi$ are proposed in Theorem \ref{Main},
then there exists a $C^{4, \alpha}$, $k$-convex, star-shaped and closed hypersurface $\Sigma$ in $\{(r,u)\in \overline{M}\mid r_{1}\leq r \leq r_{2}\}$ that satisfies equation \eqref{Eq} for  $\alpha\in (0,1)$.
\end{corollary}


The organization of the paper is as follows.
In Sect. 2
we start with some preliminaries.
$C^0$, $C^1$ and $C^2$ estimates are given in Sect. 3.
In Sect. 4 we prove theorem \ref{Main}.

After we completed our paper, we found that Wang independently proved the corresponding curvature estimates for $k=n-1, n-2$ in Theorem 4.1 of \cite{W24}. It also provides a new perspective to prove the global curvature estimates.

\section{Preliminaries}

\subsection{Star-shaped hypersurfaces in the warped product
manifold}

Let $M$ be a compact Riemannian manifold with the metric $g'$ and $I$ be an open interval in $\mathbb{R}$. Assuming $\lambda : I\rightarrow \mathbb{R}^{+}$ is a positive differential function and
$\lambda'>0$, the manifold $\overline{M}=I\times_{\lambda}M$ is called the warped product if it is endowed with the metric
\begin{eqnarray*}
\overline{g}^{2}=dr^2+\lambda^2(r) g'.
\end{eqnarray*}
The metric in $\overline{M}$ is denoted by $\langle\cdot,\cdot\rangle$. The corresponding Riemannian connection in $\overline{M}$ will be denoted by $\overline{\nabla}$. The usual connection in $M$ will be denoted by $\nabla'$. The curvature tensors in $M$ and $\overline{M}$ will be denoted by $R$ and $\overline{R}$, respectively.

Let $\{e_1,\cdots,e_{n-1}\}$ be an orthonormal frame field in M and let $\{\theta_1,\cdots,\theta_{n-1}\}$ be the associated dual frame.
The connection forms $\theta_{ij}$ and curvature forms $\Theta_{ij}$ in M satisfy the structural equations
  \begin{align}
    &d\theta_i=\sum_j\theta_{ij}\wedge\theta_j,\quad \theta_{ij}=-\theta_{ji} ,\\
    &d\theta_{ij}-\sum_k\theta_{ik}\wedge\theta_{kj}=\Theta_{ij}=-\frac{1}{2}\sum_{k,l}R_{ijkl}\theta_k
    \wedge\theta_l.
  \end{align}
An orthonormal frame in $\overline{M}$ may be defined by $\overline{e}_i=\frac{1}{\lambda}e_i,1\leq i\leq n-1$, and $\overline{e}_0=\frac{\partial}{\partial r}$. The associated dual frame is that $\overline{\theta}_i=\lambda\theta_i$ for $1\leq i \leq n-1$ and $\overline{\theta}_0=dr$.
Then we have the following lemma (See \cite{Hu20}).

\begin{lemma}
Given a differentiable function $r : M \rightarrow I$, its graph is defined by the hypersurface
\begin{eqnarray*}
\Sigma=\{(r(u),u): u \in M\}.
\end{eqnarray*}
Then the tangential vector takes the
form
$$X_i=\lambda\overline{e}_i+r_i \overline{e}_0,$$
where $r_i$ are the components of the differential $dr=r_i \theta^i$.
The induced metric on $\Sigma$ has
$$g_{ij}=\lambda^2(r)\delta_{ij}+r_i r_j,$$
and its inverse is given by
$$g^{ij}=\frac{1}{\lambda^2}(\delta_{ij}-\frac{r^i r^j}{v^2}).$$
We also have the outward unit normal vector of $\Sigma$
$$\nu=-\frac{1}{v}\bigg(\lambda \overline{e}_0 -r^i\overline{e}_i\bigg),$$
where $v=\sqrt{\lambda^2+|\nabla'r|^2}$ with $\nabla'r=r^ie_i$.
Let $h_{ij}$ be the second fundamental form of $\Sigma$ in term of the tangential vector fields $\{X_1,
..., X_n\}$. Then,
$$h_{ij}=-\langle\overline{\nabla}_{X_j} X_i, \nu\rangle=
\frac{1}{v}\bigg(-\lambda r_{ij}+2\lambda'r_ir_j+\lambda^2\lambda' \delta_{ij}\bigg)$$
and
$$h^i_j=\frac{1}{\lambda^2 v}(\delta_{ik}-\frac{r^i r^k}{v^2})\bigg(-\lambda r_{kj}+2\lambda'r_kr_j+\lambda^2\lambda' \delta_{kj}\bigg),$$
where $r_{ij}$ are the components of the Hessian $\nabla'^{2} r=\nabla' dr $ of $r$ in $M$.
\end{lemma}

The Codazzi equation is a commutation formula for the first order derivative of $h_{ij}$ given by
\begin{equation}\label{Ce}
  h_{ijk}-h_{ikj}=\overline{R}_{0ijk}
\end{equation}
and the Ricci identity is a commutation formula for the second order derivative of $h_{ij}$ given by

\begin{lemma}
   Let $\overline{X}$ be a point of $\Sigma$ and $\{E_{0}=\nu,E_1,\cdots,E_n\}$ be an adapted frame field such that each $E_i$ is a principal direction and  $\omega_i^k=0$ at $\overline{X}$. Let $(h_{ij})$ be the second quadratic form of $\Sigma$. Then, at the point $\overline{X}$, we have
   \begin{equation}
     h_{ii11}-h_{11ii}=h_{11}h_{ii}^2-h_{11}^2h_{ii}+2(h_{ii}-h_{11})\overline{R}_{i1i1}+h_{11}\overline{R}
     _{i0i0}-h_{ii}\overline{R}_{1010}+\overline{R}_{i1i0;1}-\overline{R}_{1i10;i}.
   \end{equation}
\end{lemma}
\begin{proof}
See \cite[Lemma 2.2]{Chen}.
\end{proof}

Consider the function
\begin{eqnarray*}
\tau=\langle V, \nu\rangle, \quad \quad  \Lambda(r)=\int_{0}^{r}\lambda(s) d s
\end{eqnarray*}
with the position vector field
\begin{eqnarray*}
V=\lambda(r)\frac{\partial}{\partial_r}.
\end{eqnarray*}
 Then we need the following lemma for $\tau$ and  $\Lambda$.
\begin{lemma}\label{supp}
\begin{eqnarray}\label{1d-lad}
\nabla_{E_i} \Lambda =\lambda \langle \overline{e}_0, E_i\rangle E_i,
\end{eqnarray}
\begin{eqnarray}\label{1d-tau}
\nabla_{E_i} \tau = \sum_j \nabla_{E_j} \Lambda h_{ij},
\end{eqnarray}
\begin{eqnarray}\label{2d-lad}
\nabla^2_{E_i, E_j} \Lambda=\lambda^{\prime}g_{ij}-\tau h_{ij}
\end{eqnarray}
and
\begin{eqnarray}\label{2d-tau}
\nabla^2_{E_i, E_j} \tau=-\tau \sum_k h_{ik}h_{kj}+ \lambda^{\prime}h_{ij}+\sum_k \left( h_{ijk}-\overline{R}_{0ijk}\right) \nabla_{E_k} \Lambda.
\end{eqnarray}
\end{lemma}
\begin{proof}
See Lemma 2.2, Lemma 2.6 and Lemma 2.3 in \cite{Guan15}, \cite{Jin} or \cite{Chen} for the proof.
\end{proof}

\subsection{$k$-th elementary symmetric functions}
Let $\kappa=(\kappa_1,\dots,\kappa_n)\in\mathbb{R}^n$, then we recall the
 definition of elementary symmetric function for $1\leq k\leq n$
\begin{equation*}
\sigma_k(\kappa)= \sum _{1 \le i_1 < i_2 <\cdots<i_k\leq
n}\kappa_{i_1}\kappa_{i_2}\cdots\kappa_{i_k}.
\end{equation*}

\begin{definition}
 A $C^2$ regular hypersurface $M\subset\mathbb{R}^{n+1}$ is called $k$-convex if its principal curvature vector $\kappa(X)\in\overline{\Gamma}_k$ for all $X\in M$. For a domain $\Omega\subset\mathbb{R}^n$, a function $u\in C^2(\Omega)$ is called admissible if its graph is $k$-convex. Here $\Gamma_k$ is the G{\aa}rding's cone
$$\Gamma_k  = \{ \kappa  \in \mathbb{R}^n :\sigma _m (\kappa ) >
0,\quad m=1,\cdots,k\}.$$
\end{definition}

Denote $\sigma_{k-1}(\kappa|i)=\frac{\partial
\sigma_k}{\partial \kappa_i}$ and
$\sigma_{k-2}(\kappa|ij)=\frac{\partial^2 \sigma_k}{\partial
\kappa_i\partial \kappa_j}$, then we list some properties of
$\sigma_k$ which will be used later.

\begin{lemma}\label{n-2-pre-lem1}
If $\kappa\in\Gamma_k$ and $\kappa_1\geq\cdots\geq\kappa_k\geq\cdots\geq\kappa_n$, then we have

(a) For any $1\leq l<k$, we have
\[ \sigma_l(\kappa)\geq\kappa_1\kappa_2\cdots\kappa_l,\]

(b)
\[\sigma_k(\kappa)\leq C_n^k\kappa_1\cdots\kappa_k,\]

(c)
\[\sigma_{k-1}(\kappa|k)\geq C(n,k)\sigma_{k-1}(\kappa),
\]

(d)
\[-\kappa_i<\frac{(n-k)\kappa_1}{k},\]
if $\kappa_i\leq0$, $1\leq i\leq n$,

(e)
\[\sum_i\sigma_{k-1}(\kappa|i)\kappa_i^2\geq\frac{k}{n}\sigma_1(\kappa)\sigma_k(\kappa).\]
\end{lemma}
\begin{proof}
See Proposition 1.2.7, 1.2.9, Corollary 1.2.11 in \cite{CQ1}, Lemma 2.2 in \cite{Lu} and Lemma 8, 9 in \cite{Ren2} for the proof.
\end{proof}


\begin{lemma}\label{25}
Assume that $\kappa=(\kappa_1,\cdots,\kappa_n)\in\Gamma_k$. Then for any given indices $1\leq i,j\leq n$, if $\kappa_i\geq\kappa_j$, we have
$$|\sigma_{k-1}(\kappa|ij)|\leq\sqrt{\frac{k(n-k)}{n-1}}\sigma_{k-1}(\kappa|j).$$
\end{lemma}
\begin{proof}
See Lemma 6 in \cite{Ren2} and the proof was given in \cite{LT94}.
\end{proof}

\begin{lemma}\label{r2}
Let $\kappa=(\kappa_1,\cdots,\kappa_n)\in\Gamma_k$ with $\kappa_1\geq\kappa_2\geq\cdots\geq\kappa_n$ and $n<2k$. Assume that $\sigma_k(\kappa)\geq N_0>0$. Then for any $1\leq i,j\leq n$ with $i\neq j$, if $\kappa_i\geq\kappa_1-\frac{\sqrt{\kappa_1}}{n}$, we have
$$\frac{2\kappa_i(1-e^{\kappa_j-\kappa_i})}{\kappa_i-\kappa_j}\sigma_k^{jj}(\kappa)\geq\sigma_k^{jj}(\kappa)
+(\kappa_i+\kappa_j)\sigma_k^{ii,jj}(\kappa),$$
when $\kappa_1$ is sufficiently large.
\end{lemma}
\begin{proof}
See Lemma 13 in \cite{Ren2}.
\end{proof}

\begin{lemma}\label{tu}
For any $\epsilon\in(0,1)$, there exists a positive constant $\delta<4\epsilon$ such that the function $f(x)=x-(1-\epsilon)(1-e^{-x})(x+\delta)>0$ for any $x\in (0,+\infty)$.
\end{lemma}
\begin{proof}
See Lemma 2.9 in \cite{Tu}.
\end{proof}

\section{The  priori estimates}

In order to prove Theorem \ref{Main}, we use the degree theory for
nonlinear elliptic equation developed in \cite{Li89} and the proof
here is similar to \cite{Li-Ol, Jin, An, Li-Sh}. First, we consider
the family of equations for $0\leq t\leq 1$, $n<2k$
\begin{eqnarray}\label{Eq2}
\sigma_k(h_{j}^i)=\widetilde{\psi},
\end{eqnarray}
where $\widetilde{\psi}=t\psi(r,u, \nu)+(1-t)\phi(r) C_n^k\zeta^k(r)$, $\zeta(r)=\lambda'(r)/\lambda(r)$
and $\phi$ is a positive function which satisfies the following
conditions:

(a) $\phi(r)>0$,

(b) $\phi(r)\geq 1$ for $r\leq r_1$,

(c) $\phi(r)\leq 1$ for $r\geq r_2$,

(d) $\phi^{\prime}(r)<0$.

\subsection{$C^0$ Estimates}
Now, we can prove the following proposition which asserts that  the
solution of equation \eqref{Eq2} have uniform $C^0$ bound.

\begin{proposition}\label{n-2-C^0}
Under the assumptions \eqref{ASS1} and \eqref{ASS2} mentioned in
Theorem \ref{Main}, if the $k$-convex hypersurface $\overline{\Sigma}=\{(r(u),u)\mid
u \in M \}\subset \overline{M}$ satisfies the equation
\eqref{Eq2} for a given $t \in (0, 1]$, then
\begin{eqnarray*}
r_1<r(u)<r_2, \quad \forall \ u \in M.
\end{eqnarray*}
\end{proposition}

\begin{proof}
Assume $r(u)$ attains its maximum at $u_0 \in M$ and
$r(u_0)\geq r_2$, then recalling
\begin{eqnarray*}
h^i_j=\frac{1}{\lambda^2 v}(\delta_{ik}-\frac{r^i r^k}{v^2})\bigg(-\lambda r_{kj}+2\lambda'r_kr_j+\lambda^2\lambda' \delta_{kj}\bigg),
\end{eqnarray*}
which implies together with the fact the matrix $r_{ij}$ is
non-positive definite at $u_0$
\begin{eqnarray*}
h^{i}_{j}(u_0)=\frac{1}{\lambda^3}\bigg(-\lambda r_{ij}+\lambda^2\lambda' \delta_{ij}\bigg)\geq \frac{\lambda'}{\lambda} \delta_{ij}.
\end{eqnarray*}
Thus, we have at $u_0$
\begin{eqnarray*}
\sigma_k(h_{j}^{i})\geq C_n^k \zeta^k(r).
\end{eqnarray*}
So, we arrive at $u_0$
\begin{eqnarray*}
t\psi(r,u, \nu)+(1-t)\phi(r) C_n^k\zeta^k(r)\geq C_n^k\zeta^k(r).
\end{eqnarray*}
Thus, we obtain at $u_0$
\begin{eqnarray*}
\psi(r,u, \nu)\geq C_n^k\zeta^k(r),
\end{eqnarray*}
which is in contradiction with \eqref{ASS2}. Thus, we have $r(u)<
r_2$ for $u \in M$. Similarly, we can obtain $r(u)> r_1$
for $u \in M$.
\end{proof}

Now, we prove the following uniqueness result.

\begin{proposition}\label{Uni}
For $t=0$, there exists an unique $k$-convex solution of the
equation \eqref{Eq2}, namely $\Sigma_0=\{(r(u),  u) \in \overline{M} \mid
r(u)=r_0\}$, where $r_0$ satisfies $\varphi(r_0)=1$.
\end{proposition}

\begin{proof}
Let $\Sigma_0$ be a solution of \eqref{Eq2} for $t=0$, then
\begin{eqnarray*}
\sigma_k(h_{j}^i)-\phi(r) C_n^k\zeta^k(r)=0.
\end{eqnarray*}
Assume $r(u)$ attains its maximum $r_{max}$ at $u_0 \in
M$, then we have at $u_0$
\begin{equation*}
h^{i}_{j}=\frac{1}{\lambda^3}\bigg(-\lambda r_{ij}+\lambda^2\lambda' \delta_{ij}\bigg),
\end{equation*}
which implies together with the fact the matrix $r_{ij}$ is
non-positive definite at $u_0$
\begin{eqnarray*}
\sigma_k(h_{j}^{i})\geq C_n^k \zeta^k(r).
\end{eqnarray*}
Thus, we have by the equation \eqref{Eq2}
\begin{eqnarray*}
\varphi(r_{max})\geq 1.
\end{eqnarray*}
Similarly,
\begin{eqnarray*}
\varphi(r_{min})\leq 1.
\end{eqnarray*}
Thus, since $\varphi$ is a decreasing function, we obtain
\begin{eqnarray*}
\varphi(r_{min})=\varphi(r_{max})=1.
\end{eqnarray*}
We conclude
\begin{eqnarray*}
r(u)=r_0
\end{eqnarray*}
for any  $( r(u), u) \in \overline{M}$, where $r_0$ is the unique solution of
$\varphi(r_0)=1$.
\end{proof}

\subsection{$C^1$ Estimates}
In this section, we establish gradient estimates for equation \eqref{Eq2}.

\begin{theorem}\label{n-2-C1e}
Under the assumption \eqref{ASS3}, if the closed star-shaped  $k$-convex hypersurface $\Sigma=\{(r(u),u)\in \overline{M}\mid u \in M\}$ satisfying the curvature equation \eqref{Eq2}
 and $\lambda$ has positive upper and lower
bound. Then there exists a constant C depending only on $n, k, \|\lambda\|_{C^1}, \inf r,  \sup r, \inf \widetilde{\psi}$, $\|\widetilde{\psi}\|_{C^1}$ and the curvature $\overline{R}$ such that
\begin{equation*}
|\nabla r|\leq C.
\end{equation*}
\end{theorem}

\begin{proof}
As the treatment in \cite{Chen}, it is sufficient to obtain a positive lower bound of $\tau$. If $V$ is parallel to the normal direction $\nu$ of at $u_0$, we can obtain the lower bound of $\tau$. Thus, our result holds. So we assume $V$ is not parallel to the normal direction $\nu$ at $u_0$ and derive a contradiction. More details can refer to Lemma 3.1 in \cite{Chen}.
\end{proof}

\subsection{$C^2$ Estimates}

  Under the assumptions \eqref{ASS1}-\eqref{ASS3},  from Theorem   \ref{n-2-C^0} and  \ref{n-2-C1e} we know that
there exists a positive constant $C_0$ depending on $\inf_{\Sigma} r$ and $\|r\|_{C^1}$ such that
$$\frac{1}{C_0} \leq inf_{\Sigma} \tau \leq
\tau \leq \sup_{\Sigma} \tau \leq C_0.$$

\begin{theorem}\label{n-2-C2e}
Let $\Sigma$ be a closed star-shaped  $k$-convex hypersurface satisfying equation \eqref{Eq2}  and the assumptions of Theorem \ref{Main} with $n<2k$. Then there exists a constant C depending only on $n, k, \|\lambda\|_{C^1}, \|r\|_{C^1}, \inf\lambda', \inf r,  \sup r, \inf \widetilde{\psi}$, $\|\widetilde{\psi}\|_{C^1}$ and the curvature $\overline{R}$ such that for $1\leq i\leq n$
\begin{equation*}
|\kappa_{i}(u)|\le C, \quad \forall ~ u \in M.
\end{equation*}
\end{theorem}

\begin{proof}
Taking the allxillary function
\begin{equation*}
Q=\log\kappa_1- A\tau+ B\Lambda,
\end{equation*}
where $A, B >1$ are constants to be determined later.
Suppose $Q$ attains its maximum at $V_0$. We can choose a local orthonormal frame  $\{E_{1}, E_{2}, \cdots, E_{n}\}$ near $V_0$ such that $(h_{ij})$ is diagonalized. Without loss of generality, we may assume $\kappa_1$ has multiplicity $m$, then
$$h_{ij}=\kappa_i\delta_{ij}, \quad \kappa_1=\cdots=\kappa_m>\kappa_{m+1}\geq\cdots\geq\kappa_n\quad \mbox{at} ~V_0.$$
As the perturbation argument in \cite{Chu21}, we need to perturb $h_{ij}$ by a diagonal matrix $T$ which satisfies
$$T_{ij}=\delta\delta_{ij}(1-\delta_{1i}),\quad T_{ij,p}=T_{11,ii}=0\quad \mbox{at}~V_0,$$
$\delta<1$ is a sufficiently small constant to be determined later. Thus we define $\widetilde{h}_{ij}=h_{ij}-T_{ij}$ and denote its eigenvalues by $\widetilde{\kappa}_1\geq\widetilde{\kappa}_2\geq\cdots\geq\widetilde{\kappa}_n$. It then follows that $\kappa_1\geq\widetilde{\kappa}_1$ near $V_0$ and
\begin{eqnarray*}
  \widetilde{\kappa}_i=
  \begin{cases}
  \kappa_1, ~&\mbox{if}~i=1,\\
  \kappa_i-\delta, ~ &\mbox{if}~i>1,
  \end{cases}
  \quad \mbox{at}~V_0.
\end{eqnarray*}
Thus $\widetilde{\kappa}_1>\widetilde{\kappa}_2$ at $V_0$, then $\widetilde{\kappa}_1$ is smooth at $V_0$. We consider the new function
$$\widetilde{Q}=\log\widetilde{\kappa_1}-A\tau+B\Lambda.$$
It still attains its maximum at $V_0$. Since $\widetilde{\kappa}_1=\kappa_1$ at $V_0$, then at $V_0$ we have
\begin{eqnarray}\label{n-2-c2-1}
0=\widetilde{Q}_i=\frac{\widetilde{\kappa}_{1,i}}{\widetilde{\kappa}_1}-A\tau_i+B\Lambda_i=\frac{\widetilde{\kappa}_{1,i}}{\widetilde{\kappa}_1}
-A\sum_jh_{ij}\Lambda_j+B\Lambda_i,
\end{eqnarray}
and
\begin{eqnarray}\label{n-2-c2-2}
\nonumber0\geq\sigma_k^{ii}\widetilde{Q}_{ii}&=&\sigma_k^{ii}(\log\widetilde{\kappa}_1)_{ii}-A\sigma_k^{ii}\tau_{ii}
+B\sigma_k^{ii}\Lambda_{ii}\\
\nonumber&=&\sigma_k^{ii}(\log\widetilde{\kappa}_1)_{ii}-A\sigma_k^{ii}\{-\tau h_{ii}^2+\lambda'h_{ii}+\sum_l(h_{iil}-\overline{R}_{0iil})\Lambda_l\}\\
&&+B\sigma_k^{ii}(\lambda'g_{ii}-\tau h_{ii}).
\end{eqnarray}
Note that
\begin{eqnarray}\label{n-2-c2-20}
h_{11ii}&=&h_{ii11}+h_{11}^2h_{ii}-h_{ii}^2h_{11}+\overline{R}_{0ii1;1}+\overline{R}_{01i1;i}+h_{i1}\overline{R}_{0i01}+h_{1i}\overline{R}_{01i0}\\
\nonumber&&-2h_{11}\overline{R}_{1ii1}+h_{11}\overline{R}_{0ii0}-2h_{ii}\overline{R}_{i1i1}+h_{ii}\overline{R}_{0101}.
\end{eqnarray}
We divide our proof in three steps. For convenience, we will use a unified notation $C$ to denote a constant depending on $n, k, \|\lambda\|_{C^1}, \|r\|_{C^1}, \inf\lambda', \inf r,  \sup r, \inf \widetilde{\psi}$, $\|\widetilde{\psi}\|_{C^2}$ and the curvature $\overline{R}$.

\textbf{Step 1}:  We show that
\begin{eqnarray}\label{n-2-c2-99}
\nonumber0&\geq&-\frac{1}{\kappa_1}\sum_{p\neq q}\sigma_k^{pp,qq}h_{pp1}h_{qq1}+2\sum_{p>1}\frac{\sigma_k^{pp}h_{1pp}^2}{\kappa_1(\kappa_1-\widetilde{\kappa}_p)}
-\frac{\sigma_k^{11}h_{111}^2}{\kappa_1^2}-\frac{C}{\kappa_1}\sum_{p>m}\sigma_k^{11,pp}\\
&&+(A\tau-1)\sigma_k^{ii}h_{ii}^2+(B\lambda'-CA-\frac{C}{\delta\kappa_1})\sum_i\sigma_k^{ii}-C(h_{11}+A+B).
\end{eqnarray}

The following calculations are all at $V_0$.
By Lemma 3.1 in \cite{Chu21}, we know that
$$\widetilde{\kappa}_{1,i}=h_{11i},\quad \widetilde{\kappa}_{1,ii}=h_{11ii}+2\sum_{p>1}\frac{h_{1pi}^2}{\kappa_1-\widetilde{\kappa}_p}.$$
Differentiating \eqref{Eq2} twice, we obtain
\begin{eqnarray}\label{n-2-c2-3}
\sigma_k^{ii} h_{iij}=d_V \widetilde{\psi}(\nabla_j V)+d_{\nu} \widetilde{\psi} (\nabla_j \nu)=\lambda' d_V \widetilde{\psi}(E_j)+ h_{jl} d_{\nu} \widetilde{\psi} (E_l)
\end{eqnarray}and
\begin{align}\label{n-2-c2-4}
&\sigma_k^{ii} h_{ii11}+\sigma_k^{ij,pq} h_{ij1} h_{pq1}\\
\nonumber=&d_V \widetilde{\psi}(\nabla_{11} V)+d^2_V \widetilde{\psi}(\nabla_1 V, \nabla_1 V)+2d_V d_{\nu}\widetilde{\psi}(\nabla_1 V, \nabla_1 \nu)+d^2_{\nu} \widetilde{\psi} (\nabla_1 \nu, \nabla_1 \nu)+d_{\nu} \widetilde{\psi} (\nabla_{11} \nu)\\
\nonumber\geq&-C-Ch_{11}^2+\sum_l h_{l11} d_{\nu} \widetilde{\psi} (E_l).
\end{align}
Without loss of generality, we assume that $\kappa_1\geq1$, then by \eqref{n-2-c2-20} and \eqref{n-2-c2-4}
\begin{eqnarray}\label{e1}
 \nonumber \sigma_k^{ii}(\log\widetilde{\kappa}_1)_{ii} &=& \frac{\sigma_k^{ii}\widetilde{\kappa}_{1,ii}}{\widetilde{\kappa}_1}-\frac{\sigma_k^{ii}\widetilde{\kappa}_{1,i}^2}{\widetilde{\kappa}_1^2} \\
 \nonumber &=&\frac{\sigma_k^{ii}h_{11ii}}{\kappa_1}+2\sum_{p>1}\frac{\sigma_k^{ii}h_{1pi}^2}{\kappa_1(\kappa_1-\widetilde{\kappa}_p)}
  -\frac{\sigma_k^{ii}h_{11i}^2}{\kappa_1^2} \\
 \nonumber &\geq&-\frac{\sigma_k^{ij,pq}h_{ij1}h_{pq1}}{\kappa_1}+2\sum_{p>1}\frac{\sigma_k^{ii}h_{1pi}^2}{\kappa_1(\kappa_1-\widetilde{\kappa}_p)}
  +\frac{1}{\kappa_1}\sum_lh_{l11}d_{\nu}\widetilde{\psi}(E_l)\\
  &&-\frac{\sigma_k^{ii}h_{11i}^2}{\kappa_1^2}-\sigma_k^{ii}h_{ii}^2-C\sum_i\sigma_k^{ii}-Ch_{11}-C.
\end{eqnarray}
Combining \eqref{n-2-c2-1}, \eqref{n-2-c2-3} and Codazzi equation, we get
\begin{eqnarray}\label{e2}
 \nonumber\frac{1}{\kappa_1}\sum_lh_{l11}d_{\nu}\widetilde{\psi}(E_l)&=&\frac{1}{\kappa_1}\sum_l(h_{11l}+\overline{R}_{01l1})d_{\nu}\widetilde{\psi}(E_l) \\
 \nonumber &=&\sum_l(A\sum_jh_{lj}\Lambda_j-B\Lambda_l)d_{\nu}\widetilde{\psi}(E_l)+\sum_l\frac{1}{\kappa_1}\overline{R}_{01l1}d_{\nu}\widetilde{\psi}(E_l)\\
  \nonumber&=&A\sum_j\sigma_k^{ii}h_{iij}\Lambda_j-A\lambda'\sum_jd_V\widetilde{\psi}(E_j)\Lambda_j-B\sum_l\Lambda_ld_{\nu}\widetilde{\psi}(E_l)\\
  &&+\sum_l\frac{1}{\kappa_1}\overline{R}_{01l1}d_{\nu}\widetilde{\psi}(E_l).
\end{eqnarray}
Putting \eqref{e1}-\eqref{e2} into \eqref{n-2-c2-2}, we obtain
\begin{eqnarray}\label{im}
  \nonumber0 &\geq& -\frac{\sigma_k^{ij,pq}h_{ij1}h_{pq1}}{\kappa_1}+2\sum_{p>1}\frac{\sigma_k^{ii}h_{1pi}^2}{\kappa_1(\kappa_1-\widetilde{\kappa}_p)}
  -\frac{\sigma_k^{ii}h_{11i}^2}{\kappa_1^2} \\
  \nonumber&&+(A\tau-1)\sigma_k^{ii}h_{ii}^2-(A\lambda'+B\tau)k\widetilde{\psi}\\
  &&+(B\lambda'-CA)\sum_i\sigma_k^{ii}-Ch_{11}-C(A+B).
\end{eqnarray}
Since
\begin{eqnarray}\label{i1}
  &&-\frac{\sigma_k^{ij,pq}h_{ij1}h_{pq1}}{\kappa_1}+2\sum_{p>1}\frac{\sigma_k^{ii}h_{1pi}^2}{\kappa_1(\kappa_1-\widetilde{\kappa}_p)} \\
  \nonumber&\geq&-\frac{1}{\kappa_1}\sum_{p\neq q}\sigma_k^{pp,qq}h_{pp1}h_{qq1}+\frac{1}{\kappa_1}\sum_{p\neq q}\sigma_k^{pp,qq}h_{pq1}^2+2\sum_{p>1}\frac{\sigma_k^{pp}h_{1pp}^2}{\kappa_1(\kappa_1-\widetilde{\kappa}_p)}
  +2\sum_{p>1}\frac{\sigma_k^{11}h_{1p1}^2}{\kappa_1(\kappa_1-\widetilde{\kappa}_p)}\\
  \nonumber&\geq&-\frac{1}{\kappa_1}\sum_{p\neq q}\sigma_k^{pp,qq}h_{pp1}h_{qq1}+\frac{2}{\kappa_1}\sum_{p>m}\sigma_k^{11,pp}h_{1p1}^2+2\sum_{p>1}\frac{\sigma_k^{pp}h_{1pp}^2}{\kappa_1(\kappa_1-\widetilde{\kappa}_p)}
  +2\sum_{p>1}\frac{\sigma_k^{11}h_{1p1}^2}{\kappa_1(\kappa_1-\widetilde{\kappa}_p)}.
\end{eqnarray}

When $\kappa_p\geq0$, then by choosing $\kappa_1\geq2\delta$, we have
$$\frac{\frac{1}{2}\kappa_1+\widetilde{\kappa}_p}{\kappa_1-\widetilde{\kappa}_p}=\frac{\frac{1}{2}\kappa_1+\kappa_p-\delta}
{\kappa_1-\kappa_p+\delta}\geq 0.$$

When $\kappa_p<0$, then by choosing $\kappa_1\geq\frac{2k\delta}{3k-2n}$, we have
$$\frac{\frac{1}{2}\kappa_1+\widetilde{\kappa}_p}{\kappa_1-\widetilde{\kappa}_p}=\frac{\frac{1}{2}\kappa_1+\kappa_p-\delta}{\kappa_1-\kappa_p+\delta}=-1+\frac{3}{2(1-\frac{\kappa_p}{\kappa_1}+\frac{\delta}{\kappa_1})}
\geq-1+\frac{3}{2(1-\frac{\kappa_p}{\kappa_1}+\frac{3}{2}-\frac{n}{k})}\geq0.$$
Hence by Cauchy-Schwarz inequality, Codazzi equation and choosing $\kappa_1\geq\max\{2\delta, \frac{2k\delta}{3k-2n}\}$, we derive
\begin{eqnarray}\label{i2}
  &&\frac{2}{\kappa_1}\sum_{p>m}\sigma_k^{11,pp}h_{1p1}^2+2\sum_{p>1}\frac{\sigma_k^{11}h_{1p1}^2}{\kappa_1(\kappa_1-\widetilde{\kappa}_p)}
   -\sum_{p>1}\frac{\sigma_k^{pp}h_{11p}^2}{\kappa_1^2}\\
  \nonumber &=& \frac{2}{\kappa_1}\sum_{p>m}\sigma_k^{11,pp}(h_{11p}+\overline{R}_{01p1})^2+2\sum_{p>1}\frac{\sigma_k^{11}(h_{11p}+\overline{R}_{01p1})^2}
   {\kappa_1(\kappa_1-\widetilde{\kappa}_p)}
   -\sum_{p>1}\frac{\sigma_k^{pp}h_{11p}^2}{\kappa_1^2}\\
  \nonumber&\geq&\frac{3}{2}\sum_{p>m}\frac{(\sigma_k^{pp}-\sigma_k^{11})h_{11p}^2}{\kappa_1(\kappa_1-\kappa_p)}
  +\frac{3}{2}\sum_{p>1}\frac{\sigma_k^{11}h_{11p}^2}{\kappa_1(\kappa_1-\widetilde{\kappa}_p)}-\sum_{p>1}\frac{\sigma_k^{pp}h_{11p}^2}{\kappa_1^2}\\
 \nonumber&& -\frac{C}{\kappa_1}\sum_{p>m}\sigma_k^{11,pp}
  -\frac{C}{\delta\kappa_1}\sum_i\sigma_k^{ii}\\
 \nonumber 
  \nonumber&\geq&\sum_{p>m}\sigma_k^{11}\frac{h_{11p}^2}{\kappa_1^2}\frac{\frac{1}{2}\kappa_1+\widetilde{\kappa}_p}{\kappa_1-\widetilde{\kappa}_p}
  -\frac{C}{\kappa_1}\sum_{p>m}\sigma_k^{11,pp}
  -\frac{C}{\delta\kappa_1}\sum_i\sigma_k^{ii}\\
  \nonumber&\geq&-\frac{C}{\kappa_1}\sum_{p>m}\sigma_k^{11,pp}
  -\frac{C}{\delta\kappa_1}\sum_i\sigma_k^{ii}.
\end{eqnarray}
Putting \eqref{i1}-\eqref{i2} into \eqref{im}, we obtain \eqref{n-2-c2-99}.

\textbf{Step 2}:  Next we show that
$$\sum_{p>m}\sigma_k^{11,pp}\leq C\sum_i\sigma_k^{ii}.$$
We shall discuss into two cases.

Case 1. If $\sigma_{k-1}\geq \sigma_{k-2}$. Accorrding to Lemma \ref{25}, since $\kappa_1\geq\kappa_p$ for $p>m$, then we have
$$\sum_{p>m}\sigma_k^{11,pp}\leq\sum_{p>m}|\sigma_{k-1}(\kappa|1p)|\leq\sqrt{\frac{k(n-k)}{n-1}}\sum_p\sigma_{k-1}(\kappa|p)
\leq\sqrt{\frac{k(n-k)}{n-1}}\sum_i\sigma_k^{ii}.$$

Case 2. If $\sigma_{k-1}\leq \sigma_{k-2}$, by Lemma \ref{n-2-pre-lem1} we know that
$$\kappa_1\cdots\kappa_{k-1}\leq \sigma_{k-1} \leq \sigma_{k-2} \leq C_n^{k-2}\kappa_1\cdots\kappa_{k-2},$$
which implies that $\kappa_{k-1} \leq C$. Then we divide into two sub-cases to discuss for $p>m$. Without loss of generality, we assume that $\kappa_1\geq1$.

Subcase 2.1: If $2\kappa_p \leq \kappa_1$, then for $p>m$
$$\sigma_k^{11,pp}=\frac{\sigma_k^{pp}-\sigma_k^{11}}{\kappa_1-\kappa_p}\leq\frac{\sigma_k^{pp}-\sigma_k^{11}}{\frac{\kappa_1}{2}}\leq
2\sigma_k^{pp}\leq C\sum_i\sigma_k^{ii}.$$

Subcase 2.2: For sufficiently large $\kappa_1$, if $2\kappa_p >\kappa_1$, by $\kappa_{k-1}\leq C$, we have $m<p\leq k-1$, then by Lemma \ref{n-2-pre-lem1}
$$\sigma_k^{11,pp}=\sigma_{k-2}(\kappa|1p)\leq C\frac{\kappa_2\cdots\kappa_k}{\kappa_p}\leq C\kappa_1\cdots\kappa_{k-1}\leq C\sigma_{k-1}\leq C\sigma_{k-1}(\kappa|k)\leq C\sum_i\sigma_k^{ii}.$$

\textbf{Step 3}:
By concavity of $\sigma_k^{\frac{1}{k}}$, we get
\begin{equation}\label{be}
-\frac{\epsilon}{\kappa_1}\sum_{p\neq q}\sigma_k^{pp,qq}h_{pp1}h_{qq1}\geq-\epsilon\frac{k-1}{k}\frac{(\widetilde{\psi}_1)^2}{\widetilde{\psi}\kappa_1}\geq-C\epsilon\kappa_1.
\end{equation}
According to Conjecture \ref{key}, we have
\begin{eqnarray}
 \nonumber &&-\frac{1-\epsilon}{\kappa_1}\sum_{p\neq q}\sigma_k^{pp,qq}h_{pp1}h_{qq1}+2\sum_{p>1}\frac{\sigma_k^{pp}h_{1pp}^2}{\kappa_1(\kappa_1-\widetilde{\kappa}_p)}
  -\frac{\sigma_k^{11}h_{111}^2}{\kappa_1^2}\\
 \nonumber &\geq& -\frac{K(1-\epsilon)}{\kappa_1}(\sigma_k^{jj}h_{jj1})^2-\frac{1-\epsilon}{\kappa_1^2}\sum_{j>1}a_jh_{jj1}^2-\frac{\epsilon}{\kappa_1^2}
  \sigma_k^{11}h_{111}^2+2\sum_{p>1}\frac{\sigma_k^{pp}h_{1pp}^2}{\kappa_1(\kappa_1-\widetilde{\kappa}_p)} \\
  \nonumber&\geq&-CK(1+\kappa_1)-\epsilon\frac{\sigma_k^{11}h_{111}^2}{\kappa_1^2}+2\sum_{1<p\leq m}\frac{\sigma_k^{11}h_{1pp}^2}{\delta\kappa_1}+2\sum_{p>m}\frac{\sigma_k^{pp}h_{1pp}^2}{\kappa_1(\kappa_1-\widetilde{\kappa}_p)}\\
  &&-\frac{1-\epsilon}{\kappa_1^2}\sum_{1<p\leq m}a_ph_{pp1}^2-\frac{1-\epsilon}{\kappa_1^2}\sum_{p>m}a_ph_{pp1}^2.
\end{eqnarray}
Lemma \ref{r2} implies that $a_p=\sigma_k^{11}+2\kappa_1\sigma_k^{11,pp}\leq2\kappa_1\sigma_k^{pp}$ for $1<p\leq m$. Hence by Cauchy-Schwarz inequality and choosing $\delta\leq\frac{3}{4}$, we derive
\begin{eqnarray}
  \nonumber&&2\sum_{1<p\leq m}\frac{\sigma_k^{11}h_{1pp}^2}{\delta\kappa_1}-\frac{1-\epsilon}{\kappa_1^2}\sum_{1<p\leq m}a_ph_{pp1}^2\\
  \nonumber &\geq& \frac{3}{2}\sum_{1<p\leq m}\frac{\sigma_k^{11}h_{pp1}^2}{\delta\kappa_1}-\frac{C\sum_i\sigma_k^{ii}}{\delta\kappa_1}
   -\frac{1-\epsilon}{\kappa_1^2}\sum_{1<p\leq m}a_ph_{pp1}^2\\
 \nonumber&\geq& \sum_{1<p\leq m}\sigma_k^{11}\frac{h_{pp1}^2}{\kappa_1}\left(\frac{3}{2\delta}-2(1-\epsilon)\right)-\frac{C\sum_i\sigma_k^{ii}}{\delta\kappa_1}\\
  &\geq&-\frac{C\sum_i\sigma_k^{ii}}{\delta\kappa_1}.
\end{eqnarray}
For $p>m$, according to Lemma \ref{r2}, when $\kappa_1$ is sufficiently large, we get
$$\frac{2\kappa_1(1-e^{\kappa_p-\kappa_1})}{\kappa_1-\kappa_p}\sigma_k^{pp}\geq\sigma_k^{pp}+(\kappa_1+\kappa_p)\sigma_k^{11,pp}=a_p.$$
Then choosing $\theta=1-\sqrt{1-\epsilon}$, we have
\begin{eqnarray}
  &&2\sum_{p>m}\frac{\sigma_k^{pp}h_{1pp}^2}{\kappa_1(\kappa_1-\widetilde{\kappa}_p)}-\frac{1-\epsilon}{\kappa_1^2}\sum_{p>m}a_ph_{pp1}^2\\
  \nonumber&\geq& \sum_{p>m}\frac{\sigma_k^{pp}h_{pp1}^2}{\kappa_1}\left(\frac{2(1-\theta)}{\kappa_1-\kappa_p+\delta}
  -\frac{2(1-\epsilon)(1-e^{\kappa_p-\kappa_1})}{\kappa_1-\kappa_p}\right)-\frac{C_{\theta}}{\delta\kappa_1}\sum_i\sigma_k^{ii}\\
 \nonumber &=&\sum_{p>m}\frac{2(1-\theta)\sigma_k^{pp}h_{pp1}^2}{\kappa_1(\kappa_1-\kappa_p)(\kappa_1-\kappa_p+\delta)}(\kappa_1-\kappa_p
  -(1-\theta)(1-e^{\kappa_p-\kappa_1})(\kappa_1-\kappa_p+\delta))\\
  \nonumber&&-\frac{C_{\theta}}{\delta\kappa_1}\sum_i\sigma_k^{ii}\\
  \nonumber&\geq&-\frac{C_{\theta}}{\delta\kappa_1}\sum_i\sigma_k^{ii},
\end{eqnarray}
the last inequality comes from Lemma \ref{tu} by choosing $\delta<4\theta$. Here $C_{\theta}$ is a constant depending only on $\theta$.

Using \eqref{n-2-c2-1} and Cauchy-Schwarz inequality, we derive
\begin{eqnarray}\label{en}
  \nonumber\epsilon\frac{\sigma_k^{11}h_{111}^2}{\kappa_1^2} &=&\epsilon\sigma_k^{11}(A\sum_jh_{1j}\Lambda_j-B\Lambda_1)^2 \\
 \nonumber &\leq&C\epsilon A^2\sum_i\sigma_k^{ii}h_{ii}^2+C\epsilon B^2\sigma_k^{11}\\
  &\leq& \frac{A}{2C_0}\sum_i\sigma_k^{ii}h_{ii}^2,
\end{eqnarray}
by choosing $h_{11}\geq\frac{CB}{A}$ and $\epsilon\leq\frac{1}{CC_0A}$.
Then combining Step 1-Step 2, \eqref{be}-\eqref{en} and the fact that $\sigma_k^{ii}h_{ii}^2\geq C\kappa_1$, we obtain
\begin{eqnarray*}
  0&\geq&\left(\frac{A}{2C_0}-1\right)\sigma_k^{ii}h_{ii}^2+\left(B\lambda'-CA-\frac{C_{\theta}+C}{\delta\kappa_1}\right)
  \sum_i\sigma_k^{ii}\\
  &&-C(K+1)\kappa_1-C(A+B+K)\\
  &\geq&\kappa_1-C(A+B+K),
\end{eqnarray*}
by choosing $\kappa_1\geq\frac{C_{\theta}+C}{\delta}$, $B\geq\frac{CA+1}{\inf\limits_{r_1\leq r\leq r_2}\lambda'}$, $A\geq2C_0(\frac{1}{C}+K+2)$. Then we can derive $\kappa_1\leq C(A+B+K)$, the proof is completed.
\end{proof}

\section{The proof of  Theorem  \ref{Main}}
In this section, we use the degree theory for nonlinear elliptic
equation developed in \cite{Li89} to prove Theorem \ref{Main}. The
proof here is similar to \cite{An, Jin, Li-Sh}. So, only sketch will
be given below.

After establishing the  priori estimates in Theorem \ref{n-2-C^0},
Theorem \ref{n-2-C1e} and Theorem \ref{n-2-C2e}, we know that the
equation \eqref{Eq2} is uniformly elliptic. From \cite{Eva82},
\cite{Kry83}, and Schauder estimates, we have
\begin{eqnarray}\label{C2+}
|r|_{C^{4,\alpha}(M)}\leq C
\end{eqnarray}
for any $k$-convex solution $M$ to the equation \eqref{Eq2}, where
the position vector of $\Sigma$ is $X=(r(u), u)$ for $u \in M$.
We define
\begin{eqnarray*}
C_{0}^{4,\alpha}(M)=\{r \in
C^{4,\alpha}(M): \Sigma \ \mbox{is}
 \ k\mbox{-convex}\}.
\end{eqnarray*}
Let us consider $$F(.; t): C_{0}^{4,\alpha}(M)\rightarrow
C^{2,\alpha}(M),$$ which is defined by
\begin{eqnarray*}
F(r, u; t)=\sigma_k(h^i_j)-t\psi(r,u, \nu)-(1-t)\phi(r)C_n^k\zeta^k(r).
\end{eqnarray*}
Let $$\mathcal{O}_R=\{r \in C_{0}^{4,\alpha}(M):
|r|_{C^{4,\alpha}(M)}<R\},$$ which clearly is an open
set of $C_{0}^{4,\alpha}(M)$. Moreover, if $R$ is
sufficiently large, $F(r, u; t)=0$ has no solution on $\partial
\mathcal{O}_R$ by the priori estimate established in \eqref{C2+}.
Therefore the degree $\deg(F(.; t), \mathcal{O}_R, 0)$ is
well-defined for $0\leq t\leq 1$. Using the homotopic invariance of
the degree, we have
\begin{eqnarray*}
\deg(F(.; 1), \mathcal{O}_R, 0)=\deg(F(.; 0), \mathcal{O}_R, 0).
\end{eqnarray*}
Theorem \ref{Uni} shows that $r_0$ which satisfies $\varphi(r_0)=1$ is the unique
solution to the above equation for $t=0$. Direct calculation shows
that
\begin{eqnarray*}
F(sr_0, u; 0)= (1-\phi(sr_0)) C_n^k\zeta^k(sr_0).
\end{eqnarray*}
Then
\begin{eqnarray*}
\delta_{r_0}F(r_0, u; 0)=\frac{d}{d s}|_{s=1}F(sr_0, u;
0)=-r_0\phi'(r_0) C_n^k\zeta^k(r_0),
\end{eqnarray*}
where $\delta F(r_0, u; 0)$ is the linearized operator of $F$ at
$r_0$. Clearly, $\delta_{w} F(r_0, u; 0)$ takes the form
\begin{eqnarray*}
\delta_{w}F(r_0, u; 0)=-a^{ij}w_{ij}+b^i
w_i-\phi'(r_0)C_n^k\zeta^k(r_0)w,
\end{eqnarray*}
where $(a^{ij})$ is a positive definite matrix. Since
$-\phi'(r_0) C_n^k\zeta^k(r_0)>0,$
thus $\delta F(r_0, u; 0)$ is an invertible operator. Therefore,
\begin{eqnarray*}
\deg(F(.; 1), \mathcal{O}_R; 0)=\deg(F(.; 0), \mathcal{O}_R, 0)=\pm
1.
\end{eqnarray*}
So, we obtain a solution at $t=1$. This completes the proof of
Theorem \ref{Main}.


\end{document}